\documentclass[12pt]{amsart}
\usepackage{amssymb,amsmath,amsthm}
\newtheorem{lm}{Lemma}[section]
\newtheorem{rem}[lm]{Remark}
\newtheorem{pr}[lm]{Proposition}
\newtheorem{thm}[lm]{Theorem}

\begin{document}

\title{Maharam traces on von Neumann algebras}
\author{Vladimir Chilin}
\author{Botir Zakirov}
\address{Vladimir Chilin \\Department of Mathematics, National University of Uzbekistan\\
Vuzgorodok, 100174 Tashkent, Uzbekistan} \email{{\tt
chilin@ucd.uz}}
\address{Botir Zakirov \\Tashkent Railway Engineering Institute \\
Odilhodjaev str. 1, 100167 Tashkent, Uzbekistan} \email{{\tt
botirzakirov@list.ru}}
\date{}

\begin{abstract}
Traces $\Phi$ on von Neumann algebras with values in complex order
complete vector lattices are considered. The full description of
these traces is given for the case when $\Phi$ is the Maharam
trace. The version of Radon-Nikodym-type theorem for Maharam
traces is established.\\

\noindent{\em Mathematics Subject Classification (2000).} 28B15,
46L50\\

\noindent{\em Keywords}: von Neumann algebra, measurable operator,
vector-valued trace, order complete vector lattice,
Radon-Nikodym-type theorem.

\end{abstract}

\maketitle

\section{Introduction}

The theory of integration for measures $\mu$ with values in order
complete vector lattices has inspired  the study   of
$(bo)$-complete lattice-normed spaces $L^p(\mu)$ (see, for
example, \cite{Ku1}, 6.1.8). The spaces $L^p(\mu)$ are the
Banach-Kantorovich spaces if  the measure $\mu$ possesses the
Maharam property.  In the proof of this fact, description of
Maharam operators acting in order complete vector lattices plays
an important role  (\cite{Ku1}, 3.4.3).

The existence of the center-valued traces in finite von Neumann
algebras makes it natural to construct  the theory of integration
for traces with values in the complex order complete vector
lattice $F_\mathbb{C}=F\oplus iF.$ If the  von Neumann algebra is
commutative, then construction of $F_\mathbb{C}$-valued
integration for it is the component part for the investigation of
the properties of order continuous maps of vector lattices.

Let $M$ be a non-commutative von Neumann algebra, let
$F_\mathbb{C}$ be a von Neumann subalgebra in the center of  $M$
and let $\Phi: M\to F_\mathbb{C}$ be a trace with modularity
property: $\Phi(zx)=z\Phi(x)$ for all $z\in F_\mathbb{C},~x\in M.$
It is known that the non-commutative $L^p$-space $L^p(M,\Phi)$ is
a Banach-Kantorovich space \cite{Chil.G.1.}, \cite{Chil.K.}. In
addition, $\Phi$ possesses the Maharam property:  if $0\leq z\leq
\Phi(x),~z\in F_\mathbb{C},~0\leq x\in M,$ then there exists
$0\leq y \leq x$ such that $\Phi(y)=z $ (compare with \cite{Ku1},
3.4.1).

In the present article, we will study  the faithful normal traces
$\Phi$ on a von Neumann algebra $M$ with values in an arbitrary
complex order complete vector lattice. We give the full
description of such traces in the case when $\Phi$ is a Maharam
trace. With the help of the locally measure topology in the
algebra $S(M)$ of all measurable operators  we construct the
Banach-Kantorovich space  $L^1(M,\Phi) \subset S(M).$ We also
state the version of Radon-Nikodym-type theorem for Maharam
traces.

We use the terminology and results of the  von Neumann algebras
theory  (see \cite{SZ.}, \cite{Take1.}),  measurable operators
theory (see \cite{Seg.}, \cite{Mur_m}) and  order complete vector
lattices and Banach-Kantorovich spaces  theory (see \cite{Ku1}).

\section{Preliminaries}

Let $H$ be a Hilbert space, let $B(H)$ be the $*$-algebra of all
bounded linear operators on $H,$ and $\mathbf{1}$ be the identity
operator on $H.$ Let $M$ be a von Neumann algebra acting on $H,$
let $Z(M)$ be the center of $M$ and $P(M)$ be the lattice of all
projectors in $M$. We  denote by $P_{fin}(M)$ the set of all
finite projectors in $M.$

A densely-defined  closed linear operator $x$ (possibly unbounded)
affiliated with $M$ is said to be \emph{measurable} if there
exists a sequence $\{p_n\}_{n=1}^{\infty}\subset P(M)$ such that
$p_n\uparrow \mathbf{1}$, \ $p_n(H)\subset \mathfrak{D}(x)$ and
$p_n^\bot=\mathbf{1}-p_n \in P_{fin}(M) $ for every $n=1,2,\ldots$
(here $\mathfrak{D}(x)$ is the domain of $x$). Let us denote by
$S(M)$ the set  of all measurable  operators.

Let $x,y$ be measurable  operators. Then $x+y,~xy$ and $x^*$ are
densely-defined and preclosed. Moreover, the closures
$\overline{x+y}$ (strong sum), $\overline{xy}$ (strong product)
and $x^*$ are again measurable, and $S(M)$ is a  $*$-algebra with
respect to the strong sum, strong product, and the adjoint
operation (see \cite{Seg.}). It is clear that $M$ is a
$*$-subalgebra in $S(M)$. For any subset $A\subset S(M),$ let
$A_h=\{x\in A: x=x^*\},$ $ A_+=\{x\in A: (x\xi,\xi)\geq 0 \mbox{
for all}~ \xi\in \mathfrak{D}(x)\}.$

Let  $x\in S(M)$  and  $x=u|x|$ be the polar decomposition, where
$|x|=(x^*x)^{\frac{1}{2}},$ $u$ is a partial isometry in $B(H).$
Then $u\in M$ and $|x|\in S(M).$ If $x\in S_h(M)$ and
$\{E_\lambda(x)\}$ are the spectral  projections of $x,$ then
$\{E_\lambda(x)\}\subset P(M).$

Let $M$ be a commutative von Neumann algebra. Then $M$ admits  a
faithful semi-finite normal trace $\tau,$  and $M$ is
$*$-isomorphic to the $*$-algebra $L^\infty(\Omega,\Sigma,\mu)$ of
all  bounded complex measurable functions  with the identification
almost everywhere, where $(\Omega,\Sigma,\mu)$ is a measurable
space.  In addition, $\mu(A)=\tau(\chi_A),~A\in \Sigma.$ Moreover,
$S(M)\cong L^0(\Omega,\Sigma,\mu),$ where $L^0(\Omega,\Sigma,\mu)$
is the $*$-algebra  of all complex measurable functions  with the
identification almost everywhere \cite{Seg.}.

The locally measure topology  $t(M)$ on $L^0(\Omega,\Sigma,\mu)$
is by definition the linear (Hausdorff) topology whose fundamental
system of neighborhoods around $0$ is given by
$$
W(B,\varepsilon,\delta)=\{f\in\ L^0(\Omega,\, \Sigma,\, \mu)
\colon \hbox{ there exists a set } \ E\in \Sigma, \mbox{ such that
}
$$
$$
 \ E\subseteq B, \ \mu(B\setminus E)\leqslant\delta, \
f\chi_E \in L^\infty(\Omega,\Sigma,\mu), \
\|f\chi_E\|_{{L_\infty}(\Omega,\Sigma,\mu)}\leqslant\varepsilon\}.
$$
Here \ $\varepsilon, \ \delta $ run over all strictly positive
numbers and  $B\in\Sigma$, \ $\mu(B)<\infty.$ It is known that
$(S(M),t(M))$ is a complete topological $*$-algebra.

It is clear that zero neighborhoods $W(B,\varepsilon,\delta)$ are
closed and have the following  property: if $f\in W(B,\varepsilon,
\delta),\, g\in L^\infty(\Omega,\Sigma,\mu),
\|g\|_{L{\infty}(\Omega,\Sigma,\mu)}\leq 1,$ then $gf\in
W(B,\varepsilon,\delta).$

A net $\{f_\alpha\}$ converges to $f$ locally in measure
(notation: $f_\alpha \stackrel{t(M)}{\longrightarrow}f$) if and
only if  $f_\alpha \chi_B $ converges to $f\chi_B$ in
$\mu$-measure  for each $B\in\Sigma$ with $\mu(B)<\infty$. Thus
$\{f_\alpha\}$ remains convergent to $f$ if $\tau$ is replaced by
 another faithful semi-finite normal trace on $M.$  If  $M$ is
$\sigma$-finite, i.e. any family of nonzero mutually orthogonal
projectors from $P(M)$ is at most countable, then there exists a
faithful finite normal  trace $\tau$ on $M.$ In this case, the
topology $t(M)$ is metrizable, and convergence of a sequence $f_n
\stackrel{t(M)}{ \longrightarrow} f$ is equivalent to convergence
of $f_n$ to $f$ in  trace $\tau.$

Let now $M$ be an arbitrary finite von Neumann algebra, $\Phi_M:
M\to Z(M)$ be a center-valued trace on $M$ (\cite{SZ.}, 7.11). Let
$Z(M)\cong L^\infty(\Omega,\Sigma,\mu).$ The locally measure
topology  $t(M)$ on $S(M)$ is by definition the linear (Hausdorff)
topology whose fundamental system of neighborhoods around $0$ is
given by
$$
V(B,\varepsilon, \delta ) = \{x\in S(M)\colon \ \mbox{there exists
} \ p\in P(M), z\in P(Z(M)) $$  $$ \mbox{ such that} \ xp\in M,
\|xp\|_{M}\leqslant\varepsilon, \ z^\bot \in
W(B,\varepsilon,\delta),  \ \Phi_M(zp^\bot)\leqslant\varepsilon
z\},$$ where $\|\cdot\|_{M}$ is the $C^*$-norm in $M.$  It is
known that, $(S(M),t(M))$ is a complete topological $*$-algebra
\cite{Yead1}.

The net $\{x_\alpha\}\subset S(M)$ converges to $x\in S(M)$ in
trace $\Phi_M$ (notation: $x_\alpha
\stackrel{\Phi_M}{\longrightarrow}x$ ) if
$\Phi_M(E_\lambda^\bot(|x_\alpha-x|))\stackrel{t(Z(M))}{\longrightarrow}0$
for all $\lambda>0.$

\begin{pr}\label{2.1.} (see \cite{Mur_m}, \S 3.5) Let $M$ be a finite von Neumann algebra, $x_\alpha,~x\in S(M).$ The following conditions
are equivalent:

$(i)$ $x_\alpha \stackrel{t(M)}{\longrightarrow}x;$

$(ii)$ $x_\alpha \stackrel{\Phi_M}{\longrightarrow}x;$

$(iii)$
$E_\lambda^\bot(|x_\alpha-x|)\stackrel{t(M)}{\longrightarrow}0$
for all $\lambda>0.$
\end{pr}

Let $\tau$ be a faithful semi-finite normal trace on $M.$ An
operator $x\in S(M)$ is said to be {\em $\tau$-measurable} if
$\tau(E_\lambda ^\bot(|x|)) < \infty$ for some $\lambda>0.$ The
set $S(M,\tau)$ of all $\tau$-measurable operators is the
$*$-subalgebra in $S(M),$ in addition  $M \subset S(M,\tau).$ If
$\tau(\mathbf{1}) < \infty,$ then $S(M,\tau)=S(M).$

Denote by $t_\tau$ the locally measure topology in $S(M,\tau)$
generated by a trace $\tau$ (see, for example, \cite{Nels.}). If
$x_\alpha, x \in S(M,\tau)$ and $x_\alpha$ converges to $x$ in
topology $t_\tau$ (notation: $x_\alpha
\stackrel{\tau}{\longrightarrow}x$ ), then $x_\alpha
\stackrel{t(M)}{\longrightarrow}x$ (\cite{Mur_m}, \S 3.5). If
$\tau$ is finite, then  topologies $t(M)$ and $t_\tau$ coincide
(\cite{Mur_m}, \S 3.5). It is known that   $x_\alpha
\stackrel{\tau}{\longrightarrow}x$ if and only if
$\tau(E_\lambda^\bot(|x_\alpha-x|)) \to 0 $ for all $\lambda>0$
\cite{Fac}.

Denote by $T(M)$ the set of all nonzero finite normal traces on
the finite von Neumann algebra $M.$

\begin{pr}\label{2.5.}
Let $M$ be a finite von Neumann algebra, $x_\alpha,x\in S(M).$
Then

$(i)$ if $x_\alpha\stackrel{t(M)}{\longrightarrow}x,$ then
$|x_\alpha|\stackrel{t(M)}{\longrightarrow}|x|$ and
$\tau(E_\lambda^\bot(|x_\alpha-x|))\to 0$ for all $\lambda>0$ and
$\tau \in T(M);$

$(ii)$ if $T_1(M)$ is a separating subset of $T(M)$ and
$\tau(E_\lambda^\bot(|x_\alpha-x|))\to 0$ for all $\lambda>0,$
$\tau \in T_1(M),$ then $x_\alpha\stackrel{t(M)}{\longrightarrow}
x.$
\end{pr}
\begin{proof} $(i)$ Let $\tau\in T(M)$ and $s(\tau)$ be the support
of a trace $\tau.$ Then $s(\tau)\in P(Z(M))$ and $\tau(x)=\tau(
xs(\tau))$ for all $x\in M$ (\cite{SZ.}, 5.15, 7.13). Since
$x_\alpha\stackrel{t(M)}{\longrightarrow}x,$ $x_\alpha
s(\tau)\stackrel{t(M)}{\longrightarrow}xs(\tau).$ The restriction
of $\tau$ on $Ms(\tau)$ is a faithful finite normal trace.
Therefore $\tau(E_\lambda^\bot(|x_\alpha-x|))=\tau( E_\lambda
^\bot(|x_\alpha s(\tau)-xs(\tau)|)) \to 0$ for all $\lambda>0.$

If $|x_\alpha|\stackrel{t(M)}{\nrightarrow}|x|,$  then there are
$\lambda_0>0,$ $\tau \in T(M)$ such that $\tau(E_{\lambda_0}^\bot(
|\,|x_\alpha|-|x|\,|))\nrightarrow 0.$ The restriction $\tau_0$ of
the trace $\tau$ on $Ms(\tau)$ is a faithful finite normal trace.
Therefore convergence $x_\alpha
s(\tau)\stackrel{t(M)}{\longrightarrow}xs (\tau)$ implies
$x_\alpha s(\tau)\stackrel{\tau_0}{ \longrightarrow}xs(\tau).$
Using  continuity of the operator function $\sqrt{y},~y\in
S_+(Ms(\tau))$ \cite{Tih}, we obtain
$$
|x_\alpha|s(\tau)=\sqrt{(x_\alpha s(\tau))^*(x_\alpha s(\tau))}
\stackrel{\tau_0}{\longrightarrow}\sqrt{(x^*s(\tau))(xs(\tau))}=|x|s(\tau).
$$
Hence
$\tau(E_{\lambda_0}^\bot(|\,|x_\alpha|-|x|\,|))=\tau(E_{\lambda_0}^\bot(|\,
|x_\alpha|s(\tau)-|x|s(\tau)\,|))\to 0,$ which is not the case.

$(ii)$ Since $T_1(M)$ is the separating family  traces on $M,$
$\sup\limits_{\tau\in T_1(M)}s(\tau)=\mathbf{1}.$ Hence there is a
family $\{z_i\}_{i\in I}$ of nonzero mutually orthogonal central
projectors  such that $\sup\limits_{i\in I}z_i =\mathbf{1},$ and
for any $i\in I,$ there exists  $\tau_i \in T_1(M)$ with  $z_i
\leq s(\tau_i)$ (\cite{Vl}, chapter III, \S 2). We defined  the
faithful semi-finite normal trace on $M$ as
$\tau(x)=\sum\limits_{i \in I}\tau_i(xz_i),~x\in M.$ It is clear
that restrictions $\tau$ and $\tau_i$ coincide on $Mz_i.$ In
addition, $\tau_i(E_\lambda^\bot(|x_\alpha
z_i-xz_i|))=\tau_i(E_\lambda^\bot(|x_\alpha-x|))\to 0$ for all
$\lambda>0,~i\in I.$ Hence, $E_\lambda^\bot(|x_\alpha-x|)z_i
\stackrel{\tau}{\longrightarrow}0,$ and therefore $E_\lambda^\bot(
|x_\alpha-x|)z_i\stackrel{t(M)}{\longrightarrow}0.$

For any finite subset $\gamma \subset I,$ let $u_\gamma
=\sum\limits_{i\in \gamma} z_i.$ It is clear that $u_\gamma
\uparrow\mathbf{1}$ and  $\Phi_M(u_\gamma) \uparrow
\Phi_M(\mathbf{1}).$ Hence, $\Phi_M(u_\gamma^\bot)
\stackrel{t(Z(M))}{\longrightarrow}0,$ i.e. $u_\gamma^\bot
\stackrel{t(M)}{\longrightarrow}0.$

Let $U$ be an arbitrary neighborhood of $0$ in $(S(M),t(M)).$ We
choose $V(B,\varepsilon,\delta)$ such that $V(B,\varepsilon,
\delta)+ V(B,\varepsilon,\delta)\subset U.$ Fix $\gamma_0$ with
$(\mathbf{1}-u_{\gamma_0})\in V(B,\frac{\varepsilon}{4},\delta).$
Since $E_\lambda^\bot
(|x_\alpha-x|)u_{\gamma_0}\stackrel{t(M)}{\longrightarrow}0,$
there is an $\alpha_0$ such that $E_\lambda^\bot (|x_\alpha-x|)
u_{\gamma_0}\in V(B,\varepsilon,\delta)$ as $\alpha\geq \alpha_0.$
We have $a V(B,\frac{\varepsilon}{4},\delta) b \subset
V(B,\varepsilon, \delta),$ where $a,b \in M,~ \|a\|_{M}\leqslant
1$, $\|b\|_{M}\leqslant 1$ (see, for example, \cite{Mur_m}, \S
3.5). Hence
$$
E_\lambda^\bot(|x_\alpha-x|)=E_\lambda^\bot(|x_\alpha-x|)
u_{\gamma_0}+E_\lambda^\bot(|x_\alpha-x|)(\mathbf{1}-u_{\gamma_0})
\in
$$ $$
V(B,\varepsilon,\delta)+V(B,\varepsilon,\delta)\subset U
$$
for all $\alpha\geq \alpha_0.$ Therefore  $E_\lambda^\bot
(|x_\alpha-x|)\stackrel{t(M)}{\longrightarrow}0$ for all
$\lambda>0.$ Proposition \ref{2.1.} implies that $x_\alpha
\stackrel{t(M)}{\longrightarrow}x.$
\end{proof}

\section{ Vector lattice-valued traces}

Throughout the section,  let $M$ be an  von Neumann algebra, let
$F$ be an order complete vector lattice, and let
$F_\mathbb{C}=F\oplus iF$ be a complexification of $F.$ If
$z=\alpha + i\beta \in F_\mathbb{C},~\alpha, \beta \in F, $ then
$\overline{z}:=\alpha -i\beta,$ and $|z|:=\sup\{Re( e^{i\theta}z):
0\leq \theta < 2\pi \}$ (see\cite{Ku1}, 1.3.13).

An {\em $F_\mathbb{C}$-valued trace} on the von Neumann algebra
$M$ is a linear mapping $\Phi:M\to F_\mathbb{C}$ given
$\Phi(x^*x)=\Phi(xx^*)\geq 0$ for all $x\in M.$ It is clear that
$\Phi(M_h)\subset F, ~\Phi(M_+)\subset F_+=\{a\in F: a\geq 0\}.$ A
trace $\Phi$ is said to be {\em faithful} if the equality
$\Phi(x^*x)=0$ implies $x=0,$ {\em normal} if
$\Phi(x_\alpha)\uparrow\Phi(x)$ for every $x_\alpha,x\in
M_h,~x_\alpha \uparrow x.$

If $M$ is a finite von Neumann algebra, then  its center-valued
trace $\Phi_M:M\to Z(M)$ is an example of a $Z(M)$-valued faithful
normal trace.

Let $\Delta$ be a separating family of finite normal numerical
traces on the von Neumann algebra $M,$ $\mathbb{C}^\Delta=
\prod\limits_{\tau\in\Delta}\mathbb{C}_\tau,$ where
$\mathbb{C}_\tau=\mathbb{C}$ for all $\tau\in \Delta.$ Then
$\Phi(x)=\{\tau(x)\}_{\tau\in \Delta}$ is also an example of an
faithful normal $\mathbb{C}^\Delta$-valued trace on $M.$

Let us list some properties of the trace $\Phi:M\to F_\mathbb{C}.$

\begin{pr}\label{3.2.} $(i)$ Let $x,y,a,b \in M.$ Then

$\Phi(x^*)=\overline{\Phi(x)},~ \Phi(xy)=\Phi(yx),$
$\Phi(|x^*|)=\Phi(|x|),$

$|\Phi(axb)| \leq \|a\|_M \|b\|_M \Phi(|x|);$

$(ii)$ If $ \Phi$ is a faithful trace, then $M$ is  finite;

$(iii)$ If $x_n, x \in M$ and $\|x_n-x\|_M \to 0,$ then
$|\Phi(x_n)-\Phi(x)|$ relative uniform converges  to zero;

$(iv)$ If $M$ is a finite von Neumann algebra, then $
\Phi(\Phi_M(x))=\Phi(x)$ for all $x\in M;$

$(v)$ $\Phi(|x+y|)\leq \Phi(|x|)+\Phi(|y|)$ for all $x, y \in M.$
\end{pr}
\begin{proof} The proof of  $(i)$ and $(ii)$  is the same as for
numerical traces (see, for example, \cite{Take1.}, chapter V, \S
2).

The proof of $(iii)$ follows from the inequality
$|\Phi(x_n)-\Phi(x)| \leq \|x_n-x\|_M \Phi(\mathbf{1}).$

$(iv)$ Let $U(M)$ be the set of all unitary operators in $M.$ Then
$\Phi_M(x)$ belongs to the closure  of the convex hull $co\{uxu^*:
u\in U(M)\}$ (\cite{SZ.}, 7.11). Since
$\Phi(uxu^*)=\Phi(u^*ux)=\Phi(x),$ we get $\Phi(y)=\Phi(x)$ for
any $y\in co\{uxu^*:u\in U(M)\}.$ Therefore, because of  $(iii),$
we have $\Phi(x) =\Phi(\Phi_M(x)).$

$(v)$ Since $|x+y|\leq u|x|u^* + v|y|v^*$ for some partial
isometries $u,v$ in $M$ (see \cite{Ake}), we have, by virtue of
 $(i)$
$$
\Phi(|x+y|)\leq \Phi(u|x|u^*)+\Phi(v|y|v^*)
=\Phi(u^*u|x|)+\Phi(v^*v|y|)$$ $$\leq \Phi(|x|)+\Phi(|y|).
$$
\end{proof}

The trace $\Phi:M\to F_\mathbb{C}$ possesses  {\em the Maharam
property} if for any $x\in M_+,~0\leq f\leq \Phi(x), ~f\in F,$
there exists a positive  $y\leq x$ such that $\Phi(y)=f.$ A
faithful normal $F_\mathbb{C}$-valued trace $\Phi$ with the
Maharam property is called {\em a Maharam trace} (compare with
\cite{Ku1}, III, 3.4.1). Obviously, any faithful finite numerical
trace on $M$ is a $\mathbb{C}$-valued Maharam trace.

Let us give another examples of Maharam traces. Let $M$ be a
finite von Neumann algebra, let $\mathcal{A}$ be a von Neumann
subalgebra in $Z(M),$ and let $T:Z(M) \to \mathcal{A}$ be an
injective linear positive normal operator. If $f\in
S(\mathcal{A})$ is a reversible positive element, then
$\Phi(T,f)(x)= f T(\Phi_M(x))$ is an $S(\mathcal{A})$-valued
faithful normal trace on $M.$ In addition, if $T(ab)=aT(b)$ for
all $a\in \mathcal{A}, b\in Z(M),$ then $\Phi(T,f)$ is a Maharam
trace on $M.$

Note that if $\tau$ is a faithful normal finite numerical trace on
$M$ and $ \dim(Z(M))>1,$ then $\Phi(x)=\tau(x) \mathbf{1}$ is a
$Z(M)$-valued faithful normal trace. In addition, $\Phi$ does not
possess the Maharam property. In fact, if $p\in Z(M),~0\neq p \neq
\mathbf{1},$ then for all $y\in M_+,~y\leq \mathbf{1}$ the
relation $\Phi(y)=\tau(y)\mathbf{1}\neq \tau(y)p\leq
\Phi(\mathbf{1})$ is valid.

Let $F$ have an order unit $\mathbf{1}_F.$ Denote by $B(F)$ the
complete Boolean algebra of unitary elements with respect to
$\mathbf{1}_F,$ and by $s(a):=\sup\limits_{n\geq
1}\{\mathbf{1}_F\wedge n|a| \} \in B(F)$ the support of an element
$a\in F.$ Since $|\Phi(x)|\leq \|x\|_M \Phi(\mathbf{1})$
(Proposition \ref{3.2.}$(i)$), the inequality $s(|\Phi(x)|)\leq
s(\Phi(\mathbf{1}))$ holds for all $x\in M.$ Let therefore
$s(\Phi(\mathbf{1}))=\mathbf{1}_F.$

Let $Q$ be the Stone representation space of the Boolean algebra
$B(F).$ Let $C_\infty(Q)$ be the order complete vector lattice of
all continuous functions $a: Q\to [-\infty, +\infty]$ such that
$a^{-1}(\{\pm \infty\})$ is a  nowhere dense subset of $Q.$ We
identify $F$ with the order-dense ideal in $C_\infty(Q)$
containing  algebra $C(Q)$ of all continuous real functions on
$Q.$ In addition, $\mathbf{1}_F$ is identified with the function
equal to 1  identically on $Q$ (\cite{Ku1}, 1.4.4).

The next theorem gives the description  of Maharam traces on von
Neumann algebras.

\begin{thm} \label{3.3.} Let $\Phi$ be an $F_\mathbb{C}$-valued Maharam
trace on a von Neumann algebra $M.$ Then there exists a von
Neumann subalgebra $\mathcal{A}$ in $Z(M),$ a $*$-isomorphism
$\psi$ from $\mathcal{A}$ onto the $*$-algebra $C(Q)_\mathbb{C},$
an injective  positive linear  normal operator $\mathcal{E}$ from
$Z(M)$ onto $\mathcal{A}$ with
$\mathcal{E}(\mathbf{1})=\mathbf{1},~\mathcal{E}^2=\mathcal{E},$
such that

$1)$ $\Phi(x)=\Phi(\mathbf{1})\psi(\mathcal{E}(\Phi_M(x)))$ for
all $x \in M;$

$2)$ $\Phi(zy)=\Phi(z\mathcal{E}(y))$ for all $z,y \in Z(M);$

$3)$ $\Phi(zy)=\psi(z)\Phi(y)$ for all $z\in \mathcal{A},$ $y\in
M.$
\end{thm}
\begin{proof} Since  $s(\Phi(\mathbf{1}))=\mathbf{1}_F,$ we
get that $\Phi_1(x)=\Phi(\mathbf{1})^{-1}\Phi(x)$ is  a
$(C(Q))_\mathbb{C}$-valued Maharam trace on $M.$  In addition,
$\Phi_1(\mathbf{1})=\mathbf{1}_F.$

The set $Z_h(M)$ is an order complete vector lattice with a strong
unit $\mathbf{1}$ with respect to algebraic operations, and the
partial order induced from $M_h.$ Moreover, the Boolean algebra of
all unitary elements in $Z_h(M)$ with respect to $\mathbf{1}$
coincides with $P(Z(M)).$ Let $T$ be a restriction of $\Phi_1$ on
$Z_h(M).$ Since $|\Phi_1(x)| \leq \|x\|_M,$ $T(Z_h(M))\subset
C(Q).$ It is clear that $T$ is an  injective positive order
continuous linear operator. If $x\in Z_+(M),$ $0\leq a \leq
Tx=\Phi_1(x),$ $a\in C(Q),$ then there exists $y\in M_+$ such that
$y\leq x$ and $\Phi_1(y)=a.$ By Proposition \ref{3.2.} $(iv),$ we
have $a=\Phi_1(y)=\Phi_1(\Phi_M(y)) =T(\Phi_M(y)),$ moreover,
$0\leq \Phi_M(y)\leq \Phi_M(x) =x.$ Hence, $T: Z_h(M)\to C(Q)$ is
a Maharam operator (\cite{Ku1}, 3.4.1).  Theorem 3.4.3 from
\cite{Ku1} guarantees the existence of a Boolean isomorphism
$\varphi$ from $B(F)$ onto a regular Boolean subalgebra $B$ in
$P(Z(M))$ such that $gT(x)=T(\varphi(g)x)$ for all $g\in B(F)$ and
$x\in Z_h(M).$ We denote by $\mathcal{A}$ a commutative von
Neumann subalgebra in $Z(M)$ generated by $B,$ i.e. $\mathcal{A}$
coincides with the bicommutant of $B.$  It is known that
$\mathcal{A}_h=\{x\in Z_h(M): E_\lambda(x) \in B \mbox{ for all}
~\lambda\}$ where $\{E_\lambda (x)\}$ are the spectral
 projections of $x.$  The Boolean isomorphism $\varphi$ is
extended  to the $*$-isomorphism $\widetilde{\varphi}$ from the
$*$-algebra $C(Q)_\mathbb{C}$ onto the von Neumann algebra
$\mathcal{A}.$ If  $a=\sum\limits_{i=1}^n \lambda_ie_i$ is a
simple element, $\lambda_i\in \mathbb{R},$ $e_i\in B(F),$
$i=1,\dots,n,$ then
$$
T(\widetilde{\varphi}(a)x)=\sum\limits_{i=1}^n\lambda_iT(\varphi(e_i)x)
=aT(x)
$$
for all $x\in \mathcal{A}_h.$ Furthermore, we note
$T(\widetilde{\varphi}(a)x)=aT(x)$ for any $a\in C(Q),$ $x\in
\mathcal{A}_h.$ This is obtained by approximating the elements
from $C(Q)$ by simple elements. Therefore,
$\Phi_1(\widetilde{\varphi}(a)x)= a \Phi_1 (x)$ for all $a\in
C(Q)_\mathbb{C},$ $x\in\mathcal{A},$  in particular,
$$\Phi_1(\widetilde{\varphi}(a))=a    \eqno (1)$$

Hence the restriction $T_0$ of the operator $T$ on $\mathcal{A}_h$
is a lattice isomorphism from $\mathcal{A}_h$ onto $C(Q).$
Therefore $T_0$ is a Maharam operator. By Theorem 4.2.9 from
\cite{Ku2}, there exists an operator of conditionally mathematical
expectation $\mathcal{E}: Z_h(M) \to \mathcal{A}_h$ satisfying the
following conditions:

$(E1)$ $\mathcal{E}$ is an  injective positive order continuous
linear operator, $\mathcal{E}^2=\mathcal{E}$ and
$\mathcal{E}(\mathbf{1})=\mathbf{1};$

$(E2)$ $T(xy)=T(x\mathcal{E}(y))$ for all $x,y \in Z_h(M);$

$(E3)$ $\mathcal{E}(zy)=z\mathcal{E}(y)$ for all $z\in
\mathcal{A}_h,~y\in Z_h(M).$

The operator $\mathcal{E}$ is  extended  to the operator
$\widetilde{\mathcal{E}}: Z(M) \to \mathcal{A}.$ It is clear that
the condition $(E1)$ is satisfied for $\widetilde{\mathcal{E}},$
the condition $(E2)$ has the form
$\Phi_1(xy)=\Phi_1(x\widetilde{\mathcal{E}}(y))$ for all $x,y \in
Z(M),$ and the condition $(E3)$ is valid for all $z\in
\mathcal{A},$ $y\in Z(M).$ The condition $(E2)$ implies that
$$
\Phi_1(y)=\Phi_1(\widetilde{\mathcal{E}}(y))~\mbox{ for all}~ y
\in Z(M). \eqno(2)
$$

Using   equalities (1), (2) and Proposition \ref{3.2.} $(iv)$, we
get
$$
\Phi_1(x)=\Phi_1(\Phi_M(x))=\Phi_1(\widetilde{\mathcal{E}}(\Phi_M(x))=\widetilde{\varphi}^{-1}(\widetilde{\mathcal{E}}(\Phi_M(x)))
\eqno (3)
$$
for any $x\in M.$

Taking in (3) $\psi=\widetilde{\varphi}^{-1}$ and letting
$\widetilde{\mathcal{E}}$ as $\mathcal{E},$ we obtain the
statement of Theorem \ref{3.3.}.
\end{proof}

Due to Theorem \ref{3.3.},  the $*$-algebra
$\mathcal{B}=C(Q)_\mathbb{C}$ is  $*$-isomorphic to a von Neumann
subalgebra in $Z(M).$ Therefore $\mathcal{B}$ is a commutative von
Neumann algebra, and  $*$-algebra $C_\infty(Q)_\mathbb{C}$ is
identified with  $*$-algebra $S(\mathcal{B}).$ In particular,
there exists a separating family of completely additive
scalar-valued measures on $B(F),$ and therefore $F$ is a
Kantorovich-Pinsker space (\cite{Ku1}, 1.4.10).

We claim that a version of Radon-Nikodym-type theorem is valid for
a Maharam trace $\Phi.$ For this, we need the space $L^1(M,\Phi)$
of operators from $S(M)$ to be integrable with respect to $\Phi.$

Let $F$ be a Kantorovich-Pinsker space and let $\Phi$ be an
$F_\mathbb{C}$-valued Maharam trace on the von Neumann algebra
$M.$ The net $\{x_\alpha\}\subset S(M)$ converges to $x\in S(M)$
with respect to the trace $\Phi$ (notation: $x_\alpha
\stackrel{\Phi}{\longrightarrow}x$) if $\Phi(E_\lambda^\bot(|
x_\alpha-x|))\stackrel{t(\mathcal{B})}{\longrightarrow}0$ for all
$\lambda>0.$

\begin{pr} \label{3.3} $x_\alpha
\stackrel{\Phi}{\longrightarrow}x$ iff $x_\alpha
\stackrel{t(M)}{\longrightarrow}x.$
\end{pr}
\begin{proof} Let $\nu$ be a faithful normal semi-finite numerical
trace on $\mathcal{B}.$  Choose  $\{e_i\}_{i\in I}$ to be a set of
nonzero mutually orthogonal projections from $P(\mathcal{B})$ with
$\sup\limits_{i\in I}e_i=\mathbf{1}_F$ and $\nu(e_i)<\infty, i\in
I.$ Set $\tau_i(x)=\nu(\Phi(x)\Phi(\mathbf{1})^{-1}e_i),~x\in
M,~i\in I.$ It is clear that $\{\tau_i\}_{i\in I}$ is a separating
family of finite traces on $M.$ Due to Proposition \ref{2.5.},
$x_\alpha\stackrel{t(M)}{\longrightarrow}x$ if and only if
$\tau_i(E_\lambda^\bot(|x_\alpha-x|))\to 0$ for all $\lambda>0,
i\in I.$ The last convergence is equivalent to convergence $\Phi(
E_\lambda^\bot(|x_\alpha-x|))\stackrel{t(\mathcal{B})}{\longrightarrow}0.$
\end{proof}

For each $x\in M,$ let $\|x\|_\Phi=\Phi(|x|).$ Proposition
\ref{3.2.} implies that $\|\cdot\|_\Phi$ is an $F$-valued norm on
$M.$ In addition, $\|x\|_\Phi=\|x^*\|_\Phi=\|\,|x|\,\|_\Phi$ and $
\|axb\|_\Phi \leq \|a\|_M\|b\|_M\|x\|_\Phi$ for all $x, a, b \in
M.$

We have  $\Phi(E_\lambda^\bot(|x_\alpha-x|))\leq
\frac{1}{\lambda}\Phi(|x_\alpha-x|),$ $\lambda>0,$ $x_\alpha, x\in
M.$ Hence  $\|x_\alpha-x\|_\Phi \stackrel{t(\mathcal{B}
)}{\longrightarrow} 0$ implies  $x_\alpha \stackrel{\Phi}{\to} x,$
and therefore $x_\alpha \stackrel{t(M)}{\to} x$ (Proposition
\ref{3.3}).

An operator  $x\in S(M)$ is said to be {\em $\Phi$-integrable} if
there exists a sequence $\{x_n\} \subset M$ such that $x_n
\stackrel{\Phi}{\to} x $ and $\|x_n-x_m\|_\Phi
\stackrel{t_l(\mathcal{B})}{\longrightarrow} 0$ as $n,m \to
\infty.$ Denote by $L^1(M,\Phi)$ the set of all $\Phi$-integrable
operators from $S(M).$ It is clear that $M \subset L^1(M,\Phi)$
and $L^1(M,\Phi)$ is a linear subset of $S(M).$ It follows from
Proposition \ref{3.2.} and \ref{3.3} that $M L^1(M,\Phi) M \subset
L^1(M,\Phi)$ and $x^*\in L^1(M,\Phi)$ for all $x\in L^1(M,\Phi).$

We now define  an $S_h(\mathcal{B})$-valued $L^1$-norm on
$L^1(M,\Phi).$

\begin{pr}\label{4.2.} If $x_n\in M, ~x_n
\stackrel{\Phi}{\to} 0,~\|x_n-x_m\|_\Phi
\stackrel{t(\mathcal{B})}{\longrightarrow} 0,$ then $\Phi(|x_n|)
\stackrel{t(\mathcal{B})}{\longrightarrow} 0.$
\end{pr}
\begin{proof} Since $\left|\,\|x_n\|_\Phi -\|x_m\|_\Phi\, \right| \leq
\|x_n  -x_m\|_\Phi,$  $\Phi(|x_n|)=\|x_n\|_\Phi$ is a Cauchy
sequence  in $(S(\mathcal{B}),t(\mathcal{B})).$ Because of the
completeness of  $*$-algebra $(S(\mathcal{B}), t(\mathcal{B})),$
there exists $f\in S_+(\mathcal{B})$ such that
$\Phi(|x_n|)\stackrel{t(\mathcal{B})}{\longrightarrow} f.$ We
claim that $f=0.$ First, we assume that  algebra $\mathcal{B}$ is
$\sigma$-finite. Then there exists a faithful normal finite
numerical trace $\nu$ on $\mathcal{B}.$ We have
$\Phi(|x_n|)\stackrel{\nu}{\longrightarrow}f$ and the sequence
$\{\Phi(|x_n|)\}$  has an $(o)$-convergent subsequence. Therefore,
as usual, we may and do assume that  the sequence
$\{\Phi(|x_n|)\}$ $(o)$-converges to $f$ in $S_h(\mathcal{B})$
(notation: $\Phi(|x_n|)\stackrel{(o)}{\longrightarrow} f$). Hence,
there exists $g=\sup\limits_{n\geq 1} \Phi(|x_n|)$ in
$S_h(\mathcal{B}).$ It is clear that
$$
\tau(x)=\nu(\Phi(x)(\mathbf{1}_F + g +\Phi(\mathbf{1}))^{-1})
\eqno(4)
$$
is a faithful normal finite numerical trace on $M.$ Since
topologies $t_\nu$ and $t(\mathcal{B})$ coincide, $\Phi(|x_n -
x_m|) \stackrel{\nu}{\longrightarrow}0.$ Therefore inequalities
$0\leq \Phi(|x_n- x_m|)\leq 2g,$ imply $\tau(|x_n-x_m|) \to 0.$ It
is known that  $(L^1(M,\tau), \|\,\|_{1,\tau})$ is complete, where
$\|x\|_{1,\tau}=\tau(|x|) $ \cite{Seg.}. Hence there exists $x\in
L^1(M, \tau) \subset S(M)$ such that $\|x - x_n\|_{1,\tau} \to 0$
and therefore, $x_n \stackrel{\tau}{ \longrightarrow} x$
\cite{Fac}.  Because of the  equality of topologies $t_\tau$ and
$t(M),$ we have $x=0.$  This means that $\tau(|x_n|)\to 0,$ i.e.
$\Phi(|x_n|)\stackrel{\nu}{ \longrightarrow}0.$

Now let $\mathcal{B}$ be a general (not necessarily
$\sigma$-finite) von Neumann algebra. For each $0\neq e\in
P(\mathcal{B}),$ we set $\Phi_e(x)= \Phi(x)e,~x\in M.$ It is clear
that $\Phi_e$ is a normal $S_h(\mathcal{B}e)$-valued trace on $M,$
which does not have, generally speaking, the faithfulness
property. A projection $s(\Phi_e)=\mathbf{1}-\sup\{p\in P(M):
\Phi_e(p)=0\}$ is called {\em the support}  trace of $\Phi_e.$ As
well as in the case of numerical traces ( see, for example,
\cite{SZ.}, 5.15, 7.13), one can establish that $s(\Phi_e)\in
P(Z(M))$ and $\Phi_e(x)=\Phi_e(xs( \Phi_e))$ is a faithful normal
$S_h(e\mathcal{B})$-valued trace on $Ms(\Phi_e).$

If $\Phi(|x_n|)\stackrel{t(\mathcal{B})}{\nrightarrow}0,$ then
there is a nonzero $\sigma$-finite projection $e\in
P(\mathcal{B})$ such that $\Phi(|x_n|)e
\stackrel{\nu}{\nrightarrow}0$ where $\nu$ is a faithful normal
finite numerical trace on $\mathcal{B}e.$ The last contradicts to
what we proved above.
\end{proof}

Let $x\in L^1(M,\Phi),$ $x_n\in M,$ $x_n \stackrel{\Phi}{
\longrightarrow}x$ and $\|x_n - x_m\|_\Phi \stackrel{t(
\mathcal{B})}{\longrightarrow}0.$ The inequality
$|\Phi(x_n)-\Phi(x_m)|\leq \Phi(|x_n-x_m|)$ and completeness of
the $*$-algebra $(S(\mathcal{B}), t(\mathcal{B}))$ guarantees the
 existence of $\widehat{\Phi}(x)\in S(\mathcal{B})$ such that
$\Phi(x_n)\stackrel{t(\mathcal{B})}{\longrightarrow}\widehat{\Phi}(x).$
Due to Proposition \ref{4.2.},  $\widehat{\Phi} (x)$ does not
depend on the choice of a sequence $\{x_n\}\subset M,$ for which
$x_n \stackrel{\Phi}{\longrightarrow}x$ and $\|x_n- x_m\|_\Phi
\stackrel{t(\mathcal{B})}{\longrightarrow}0,$ in particular,
$\widehat{\Phi}(x)=\Phi(x)$ for all $x\in M.$ The element
$\widehat{\Phi}(x)$  is called an {\em $S(\mathcal{B})$-valued
integral} of $x\in L^1(M,\Phi)$ by a trace $\Phi.$

It follows immediately from the definition of $\widehat{\Phi}$ and
Proposition \ref{3.2.}  that $\widehat{\Phi}$ is a linear mapping
from $ L^1(M,\Phi)$ into $S(\mathcal{B})$ and $
\widehat{\Phi}(xy)= \widehat{\Phi}(yx)$ for any $x\in M, y\in
L^1(M,\Phi).$ For each $x\in L^1(M,\Phi),$ we set
$\|x\|_\Phi=\widehat{\Phi}(|x|).$

\begin{thm}\label{3.5.} $(i)$ The mapping $\|\cdot\|_\Phi$ is an
$S_h(\mathcal{B})$-valued norm on $L^1(M,\Phi).$

$(ii)$ $(L^1(M,\Phi), \|\cdot\|_\Phi)$ is a Banach-Kantorovich
space.
\end{thm}
\begin{proof}$(i)$  Let $x\in L^1(M,\Phi),~x_n\in
M,~x_n\stackrel{\Phi}{\longrightarrow}x$ and
$\|x_n-x_m\|_\Phi\stackrel{t(\mathcal{B})}{\longrightarrow}0.$ It
follows from Propositions \ref{2.5.}$(i)$ and \ref{3.3} that
$|x_n|\stackrel{\Phi}{\longrightarrow}|x|.$ We claim that
$\|\,|x_n|-|x_m|\,\|_\Phi\stackrel{t(\mathcal{B})}{\longrightarrow}0.$

First, we assume that  algebra $\mathcal{B}$ is $\sigma$-finite.
Using the same trick as in the  proof of Proposition \ref{4.2.},
we can show that $\Phi(x_n) \stackrel{(o)}{\longrightarrow}
\widehat{\Phi}(x)$ in $S_h(\mathcal{B}).$ Therefore there exists
$g=\sup\limits_{n\geq 1}|\Phi(x_n)|$ in $S_h(\mathcal{B}).$
Consider a faithful normal finite numerical trace $\tau$ on $M$
defined by $(4).$ Since $\tau(| x_n -x_m|)\to 0$ as $n,m \to
\infty$ (see the proof of Proposition \ref{4.2.}), there exists
$y\in L^1(M,\tau)$ such that $\|y-x_n\|_{1,\tau}\to 0.$ Then
$x_n\stackrel{\tau}{ \longrightarrow}y,$ and therefore $x=y.$
Moreover, $|x_n| \stackrel{\tau}{\longrightarrow}|x|$ (Proposition
\ref{2.5.}$(i)$) and $\|\,|x_n|\,\|_{1,\tau} =\|x_n\|_{1,\tau} \to
\|x\|_{1,\tau}.$ It follows from (\cite{Fac}, Theorem 3.7) that
$\|\,|x|-|x_n|\, \|_{1,\tau} \to 0,$ in particular,
$\tau(\left|\,|x_n|-|x_m|\, \right|) \to 0$ as $n,m \to \infty.$
Convergence $\Phi(\left|\, |x_n|-|x_m|\,\right|)
(\mathbf{1}_F+g+\Phi(\mathbf{1}))^{-1}
\stackrel{\nu}{\longrightarrow}0$ implies $\|\,|x_n|-|x_m|
\,\|_\Phi\stackrel{t(\mathcal{B})}{\longrightarrow}0.$

Hence, $|x|\in L^1(M,\Phi)$ and $\Phi(|x_n|)
\stackrel{t(\mathcal{B})}{\longrightarrow}\widehat{\Phi}(|x|).$ In
particular, $\|x\|_\Phi=\widehat{\Phi}(|x|)\geq 0$ for all $x\in
L^1(M,\Phi).$ If $\widehat{\Phi}(|x|)=0,$ then $0\leq \|x_n
\|_\Phi=\Phi(|x_n|)\stackrel{t(\mathcal{B})}{\longrightarrow}0.$
Hence, $x_n\stackrel{\Phi}{\longrightarrow}0,$ and therefore
$x=0.$

Let now $\mathcal{B}$ be not a $\sigma$-finite algebra. Let
$\{e_i\}_{i\in I}$ be a family of nonzero mutually orthogonal
$\sigma$-finite projections in $\mathcal{B}$  with
$\sup\limits_{i\in I}e_i=\mathbf{1}_F.$ Since $\sup\limits_{i\in
I}s(\Phi_{e_i})=\mathbf{1}$ and $\widehat{\Phi}(|x|)e_i=
\widehat{\Phi}_{e_i}(|x|s(\Phi_{e_i}))\geq 0 $ for all $i\in I,$
we get $\widehat{\Phi}(|x|)\geq 0.$ Similarly, the equality
$\Phi(|x|)=0$ implies $\widehat{\Phi}_{e_i}(|x|s(\Phi_{e_i}))= 0,$
and therefore $|x|s(\Phi_{e_i})= 0$ for all $i\in I.$ Hence,
$x=0.$

Finally, we have $$\|x+y\|_\Phi \leq \|x\|_\Phi+\|y\|_\Phi,~
x,y\in L^1(M,\Phi),$$ due to the inequality $|x+y|\leq u|x|u^* +
v|y|v^*,$  $x, y \in S(M)$ (see \cite{Mur_m}, \S 2.4) and the
trick in  Proposition \ref{3.2.} $(v).$

$(ii)$ Let $x\in L^1(M,\Phi),~ x_n\in M,~ x_n
\stackrel{\Phi}{\longrightarrow}x$ and $\|x_n - x_m\|_\Phi
\stackrel{t(\mathcal{B})}{\longrightarrow}0.$ Fix $m$ and set
$y_{nm}=x_n-x_m$ for $n\geq m.$ We have $y_{nm}
\stackrel{\Phi}{\longrightarrow}x-x_m$ and $\|y_{nm}-y_{km}\|_\Phi
\stackrel{t(\mathcal{B})}{\longrightarrow}0$ as $n,k\to \infty.$
It follows from the proof of  $(i)$ that $\Phi(|y_{nm}|)
\stackrel{t(\mathcal{B})}{\longrightarrow}\widehat{\Phi}(|x-x_m|)=\|x-x_m\|_\Phi.$
Since $\Phi(|y_{nm}|)\stackrel{t(\mathcal{B})}{\longrightarrow}0$
as $n,m\to \infty,$ $\|x-x_m\|_\Phi
\stackrel{t(\mathcal{B})}{\longrightarrow}0.$

Let us now  show that any $(bo)$-Cauchy sequence in $(L^1(M,
\Phi),\|\cdot\|_\Phi)$ $(bo)$-converges.

First, we assume that $\mathcal{B}$ is a $\sigma$-finite von
Neumann algebra. Let $\{x_n\}\subset L^1(M,\Phi)$ and
$\|x_n-x_m\|_\Phi \stackrel{(o)}{\longrightarrow}0.$ Since
$\widehat{\Phi}$ is a positive mapping (see the proof of item
$(i)$), the inequality
$\widehat{\Phi}(E_\lambda^\bot(|x_n-x_m|))\leq
\frac{1}{\lambda}\widehat{\Phi}(|x_n-x_m|),$ $\lambda> 0$ is
valid. Hence, $\{x_n\}$ is a Cauchy sequence in $(S(M), t(M))$ and
therefore there exists  $x\in S(M)$ such that
$x_n\stackrel{t(M)}{\longrightarrow}x.$ Choose a system $\{U_n\}$
of closed   neighborhoods of 0 in $(S(\mathcal{B}),
t(\mathcal{B}))$  with $U_{n+1} + U_{n+1} \subset U_n,~n=1,2,\dots
.$ Due to what we proved above, for any $x_n\in L^1 (M,\Phi),$
there exists  $y_n\in M$ such that $\|x_n-y_n\|_\Phi \in U_n.$
Since $\sum\limits_{n=k+1}^m \|x_n-y_n\|_\Phi \in U_k$ for all $m
\geq k+1,$  the series $\sum\limits_{n=k+1}^\infty
\|x_n-y_n\|_\Phi $ converges in $(S(\mathcal{B}), t(M)).$ Hence,
$\|x_n-y_n\|_\Phi\stackrel{(o)}{\longrightarrow}0,$ and therefore
$\|y_n-y_m\|_\Phi\stackrel{(o)}{\longrightarrow}0.$ Also, by
Proposition \ref{3.3}, we get $x_n-y_n
\stackrel{\Phi}{\longrightarrow}0,$ and consequently
$y_n\stackrel{\Phi}{\longrightarrow}x.$ This means that $x\in
L^1(M,\Phi),$ in addition,  $\|x-y_n\|_\Phi
\stackrel{t(\mathcal{B})}{\longrightarrow}0$ and $\|y_n-y_m\|_\Phi
\stackrel{t(\mathcal{B})}{\longrightarrow} \|x-y_m\|_\Phi$ as $n
\to \infty.$

Since $\|x-y_m\|_\Phi \leq \sup\limits_{n\geq m}\|y_n-y_m\|_\Phi
\downarrow 0,$ we get $\|x-y_m\|_\Phi
\stackrel{(o)}{\longrightarrow}0$ and therefore
$\|x-x_n\|_\Phi\stackrel{(o)}{\longrightarrow}0.$

Now let $\{x_\alpha\}_{\alpha\in A}$ be an arbitrary $(bo)$-Cauchy
net in $L^1(M,\Phi),$ i.e. $\sup\limits_{\alpha,\beta \geq
\gamma}\|x_\alpha -x_\beta\|_\Phi \downarrow 0.$ We choose  a
sequence of indices $\alpha_1 \leq\alpha_2\leq \dots
\leq\alpha_n\leq \dots $ in $A$ such that
$\sup\limits_{\beta\geq\alpha_n}\|x_\beta -x_{\alpha_n}\|_\Phi \in
U_n.$ Then $\sup\limits_{n,m\geq k}\|x_{\alpha_n} -
x_{\alpha_m}\|_\Phi \in U_k,$ and therefore $\{x_{\alpha_n}\}$ is
a $(bo)$-Cauchy sequence in $L^1(M,\Phi).$ It follows from what we
proved above that there exists  $x\in L^1(M,\Phi)$ such that $\|x-
x_{\alpha_n} \|_\Phi \stackrel{(o)}{\longrightarrow}0.$ Let us
claim that $\|x-x_\alpha\|_\Phi \stackrel{(o)}{\longrightarrow}0,$
i.e. $(\sup\limits_{\alpha\geq\beta}\|x-x_\alpha\|_\Phi)\downarrow
0.$ Fix $\beta \in A$ and consider the net $\{x_\alpha\}_{\alpha
\geq \beta}.$ We construct a sequence of indices $\beta\leq\beta_1
\leq\beta_2\leq\dots $ such that $\alpha_n \leq\beta_n.$ Then
$\|x_{\beta_n}-x_{\alpha_n}\|_\Phi\in U_n,$ and therefore
$\|x_{\beta_n}-x_{\alpha_n}\|_\Phi
\stackrel{(o)}{\longrightarrow}0.$ Hence,
$\|x-x_{\beta_n}\|_\Phi\stackrel{(o)}{\longrightarrow}0$ and
$\|x_{\beta_n}-x_\beta \|_ \Phi\stackrel{(o)}{\longrightarrow}\|
x-x_\beta\|_\Phi$ as $n\to \infty.$ Thus, $\|x-x_\beta\|_\Phi\leq
\sup\limits_{n\geq 1}\|x_{\beta_n}-x_\beta\|_\Phi \leq
\sup\limits_{\alpha\geq \beta}\|x_\alpha-x_\beta\|_\Phi$ and
$\|x-x_\beta\|_\Phi \stackrel{(o)}{\longrightarrow}0.$

Let now $\mathcal{B}$ be not a $\sigma$-finite algebra and let
$\{x_\alpha\}$ be a $(bo)$-Cauchy net in $L^1(M,\Phi).$ Due to the
completeness of $(S(M), t(M)),$ there is  $x\in S(M)$ such that
$x_\alpha\stackrel{\Phi}{\longrightarrow}x.$ Let $\{e_i\}_{i\in
I}$ be the same family of projections in $\mathcal{B},$ as in the
proof of  $(i).$ It is clear that $\{x_\alpha s(\Phi_{e_i})\}$ is
a $(bo)$-Cauchy net in $L^1(Ms(\Phi_{e_i}), \Phi_{e_i}),$ and
therefore, by virtue of what we proved above, there exists
$x_i\in L^1(Ms(\Phi_{e_i}), \Phi_{e_i})$ such that $\|x_i-x_\alpha
s(\Phi_{e_i}) \|_{\Phi_{e_i}} \stackrel{(o)}{\longrightarrow}0.$
Convergence $x_\alpha
s(\Phi_{e_i})\stackrel{\Phi}{\longrightarrow}xs (\Phi_{e_i})$
implies $x_i=xs(\Phi_{e_i})$ for all $i\in I.$ Thus,
$\widehat{\Phi}(|x-x_\alpha|)e_i=\widehat{\Phi}_{e_i}(|x_i-x_\alpha|s(\Phi_{e_i}))e_i\stackrel{(o)}{\longrightarrow}0$
and $\|x-x_\alpha\|_\Phi \stackrel{(o)}{\longrightarrow}0.$

Hence, $(L^1(M,\Phi),\|\cdot\|_\Phi)$ is a $(bo)$-complete
lattice-normed space.

Now let us show that $(L^1(M,\Phi),\|\cdot\|_\Phi)$ is a
Banach-Kantorovich space, i.e. for any element $x\in L^1(M,\Phi)$
and any decomposition $\|x\|_\Phi=f_1+f_2,$ $f_1,f_2\in
S_+(\mathcal{B}),~f_1\wedge f_2=0,$ there exist $x_1,x_2\in
L^1(M,\Phi)$ such that $x=x_1+x_2$ and $\|x_i\|_\Phi=f_i,~i=1,2.$

Set $e_i=s(f_i).$ It is clear that $e_i\in P(\mathcal{B}),$
$e_1e_2=0,$ $e_1+e_2=s(\|x\|_\Phi).$ Since $\Phi$ is a Maharam
trace, we have $\Phi(y)=\Phi(\mathbf{1})\psi(\mathcal{E}
(\Phi_M(y))),$ $y\in M$ (see Theorem \ref{3.3.}). Let
$p_i=\psi^{-1}(e_i),$ $x_i=xp_i.$ Since $p_i\in P(\mathcal{A})
\subset P(Z(M)),$  $|x_i|=|x|p_i\in L^1(M,\Phi).$ We choose
$y_n\in M$ such that $y_n\stackrel{\Phi}{\longrightarrow}x$ and
$\|y_n-y_m\|_\Phi\stackrel{t(\mathcal{B})}{\longrightarrow}0.$
Then $|y_n|\stackrel{\Phi}{\longrightarrow}|x|,$
$\|\,|y_n|-|y_m|\,\|_\Phi\stackrel{t(\mathcal{B})}{\longrightarrow}0$
and $\Phi(|y_n|)\stackrel{t(\mathcal{B})}{\longrightarrow}
\widehat{\Phi}(|x|)$ (see  the proof of $(i)$). Set
$y_n^{(i)}=y_np_i,$ $i=1,2.$ We have $|y_n^{(i)}|
\stackrel{\Phi}{\longrightarrow}|x_i|$ and
$\|\,|y_n^{(i)}|-|y_m^{(i)}|\,\|_\Phi\leq
\|\,|y_n|-|y_m|\,\|_\Phi.$ Hence,
$\Phi(|y_n^{(i)}|)\stackrel{t(\mathcal{B})}{\longrightarrow}\widehat{\Phi}(|x_i|).$
Due to the property 3) from Theorem \ref{3.3.}, we have
$\Phi(|y_n^{(i)}|)=\psi(p_i)\Phi(|y_n|)=e_i\Phi(|y_n|).$

Thus, $\|x_i\|_\Phi=\widehat{\Phi}(|x_i|)=e_i\Phi(|x|)=f_i,$ in
addition $x_1+x_2=x(p_1+p_2)=x\psi^{-1}(s(\|x\|_\Phi)).$ As well
as above, one can establish that
$q\widehat{\Phi}(|x|)=\widehat{\Phi} (|x|\psi^{-1}(q))$ for all
$q\in P(\mathcal{B}).$ Taking $q=\mathbf{1}_F-s(\|x\|_\Phi),$ we
get $\widehat{\Phi}(|x|) (\mathbf{1}-\psi^{-1}(s(\|x\|_\Phi)))=0.$
Hence, $|x|=|x|\psi^{-1} (s(\|x\|_\Phi)).$ Using the polar
decomposition $x=u|x|,$ we obtain
$x=x\psi^{-1}(s(\|x\|_\Phi))=x_1+x_2.$
\end{proof}

Note another useful properties of  mapping $\widehat{\Phi}.$

Let $\Phi,~M,~Q,~\Phi_M,~\mathcal{A},~\psi$ be the same as in
Theorem \ref{3.3.}, $\mathcal{B}=C(Q)_\mathbb{C}.$ It is clear
that the $*$-isomorphism $\psi$ from $\mathcal{A}$ onto
$\mathcal{B}$ can be extended to the $*$-isomorphism from
$S(\mathcal{A})$ onto $S(\mathcal{B}).$ We denote this mapping
also by $\psi.$

\begin{pr}\label{3.6.} $S(\mathcal{A})L^1(M,\Phi) \subset
L^1(M,\Phi),$ in particular,  $S(\mathcal{A}) \subset
L^1(M,\Phi),$ in addition,
$\widehat{\Phi}(zx)=\psi(z)\widehat{\Phi} (x)$ and
$\widehat{\Phi}(\widehat{\Phi}_M(zx))=\widehat{\Phi}(zx)$ for all
$z\in S(\mathcal{A}),~x\in L^1(M,\Phi).$
\end{pr}
\begin{proof} It is sufficient to show that $x\in L_+^1(M,\Phi),~z\in S_+
(\mathcal{A})$ implies $zx\in L^1(M,\Phi)$  and
$\widehat{\Phi}(zx)=\psi(z)\widehat{\Phi}(x),$
$\widehat{\Phi}(\widehat{\Phi}_M(zx))=\widehat{\Phi}(zx).$

Let $z_n=E_n(z)z.$ It is clear that $z_n\in
\mathcal{A}_+,~z_n\uparrow z,~z_nx\in L_+^1(M,\Phi).$ Since
$z_nx=\sqrt{x}z_n\sqrt{x}\uparrow \sqrt{x}z\sqrt{x}=zx,$ we get
$$\psi(z_n)\widehat{\Phi}(x)=\widehat{\Phi}(z_nx)\leq
\widehat{\Phi}(z_{n+1}x)=\psi(z_{n+1})\widehat{\Phi}(x)\uparrow
\psi (z)\widehat{\Phi}(x).$$  Hence, $$\sup\limits_{n\geq m}
\|z_nx-z_mx\|_\Phi= \sup\limits_{n\geq m}
|\widehat{\Phi}(z_nx)-\widehat{\Phi}(z_mx)|\downarrow 0,$$ i.e.
$\{z_nx\}$ is a $(bo)$-Cauchy sequence. By Theorem \ref{3.5.},
there exists $y\in L^1(M,\Phi)$  such that
$\|z_nx-y\|_\Phi\stackrel{(o)}{\longrightarrow}0.$ The inequality
$\Phi(E_\lambda^\bot(|z_nx-y|)\leq \frac{1}{\lambda}\Phi
(|z_nx-y|)$ implies $z_nx\stackrel{\Phi}{\longrightarrow}y.$
Therefore  $y=zx,$ i.e. $zx\in L^1(M,\Phi).$  In addition,
$\psi(z_n)\widehat{\Phi}(x)=\widehat{\Phi}(z_nx)=\|z_nx\|_\Phi\stackrel{t(\mathcal{B})}{\longrightarrow}\|zx\|=\widehat{\Phi}(zx).$
Hence, $\widehat{\Phi}(zx)=\psi(z)\widehat{\Phi}(x).$

Set $x_k=E_k(x)x.$ Then ~$0\leq x_k \uparrow x,~x_k\in M.$ By
virtue of Proposition \ref{3.2.}$(iv),$
$\Phi(z_nx_k)=\Phi(\Phi_M(z_nx_k))=\Phi(z_n\Phi_M(x_k)).$ Since
$(z_nx_k)\uparrow (z_nx)$ as $k\to \infty,$ we have
$\Phi(z_nx_k)\uparrow \widehat{\Phi}(z_nx)$ and
$\Phi(\Phi_M(z_nx_k))\uparrow \widehat{\Phi}
(\widehat{\Phi}_M(z_nx)).$ Therefore $\widehat{\Phi}(z_nx)=
\widehat{\Phi}(\widehat{\Phi}_M(z_nx))$ for all $n=1,2,\dots .$
After switching to the limit as $n\to \infty,$ we obtain
$\widehat{\Phi}(zx)=\widehat{\Phi}(\widehat{\Phi}_M(zx))$
\end{proof}

Let $\Phi$ be an $F_\mathbb{C}$-valued Maharam trace on  $M$ and
let $\Psi$ be a normal $F_\mathbb{C}$-valued trace on $M.$ A trace
$\Psi$ is called {\em absolutely continuous with respect to}
$\Phi$ (notation $\Psi\ll \Phi$) if $ s(\Psi(p)) \leq s(\Phi(p))$
for all $p\in P(M).$ The last condition is equivalent to inclusion
$\Psi(p)\in \{ \Phi(p)\}^{\bot\bot}= s(\Phi(p))S_h(\mathcal{B}),$
$p\in P(M)$ where $B^\bot:=\{x\in S_h(\mathcal{B}): (\forall y\in
B) |x|\wedge |y|=0\}$ for a nonempty subset $B\subset
S_h(\mathcal{B})$ (compare with \cite{Ku1}, 6.1.11).

The next theorem is a non-commutative  version of the
Radon-Niko\-dym-type theorem for Maharam traces.

\begin{thm}\label{3.7.} Let $\Phi$ be an $F_\mathbb{C}$-valued Maharam trace on the
von Neumann algebra $M.$ If $\Psi$ is a normal
$F_\mathbb{C}$-valued trace on $M$ absolutely continuous with
respect to $\Phi,$ then there exists an operator $y\in
L_+^1(M,\Phi)\cap S(Z(M))$ such that
$$\Psi(x)=\widehat{\Phi}(yx)$$ for all $x\in M.$
\end{thm}
\begin{proof} Let $l$ be the restriction of $\Psi$ on the complete Boolean algebra
$P(Z(M)),$  and let $m$ be the restriction of $\Phi$ on $P(Z(M)).$
Obviously, $l$ and $m$ are $S_h(\mathcal{B})$-valued completely
additive measures on $P(Z(M)).$ In addition, $m(ze)=\psi(z)m(e)$
for all $z\in P(\mathcal{A}),~e\in P(Z(M))$ (see Theorem
\ref{3.3.}). Hence, $m$ is a $\psi$-modular measure on $P(Z(M))$
(see \cite{Ku1}, 6.1.9). Since the measure $l$  is absolutely
continuous with respect to $m,$ by the Radon-Nikodym-type theorem
from (\cite{Ku1}, 6.1.11), there exists $y\in
L_+^1(Z(M),m)=L_+^1(Z(M),\Phi)$ such that
$l(e)=\widehat{\Phi}(ye)$ for all $e\in P(Z(M)).$

If $a=\sum\limits_{i=1}^n \lambda_ie_i$ is a simple element from
$Z(M),$ where $\lambda_i\in \mathbb{C},~e_i\in P(Z(M)),~i=1,\dots,
n,$ then $\Psi(a)=\sum\limits_{i=1}^n \lambda_i\Psi(e_i)=
\sum\limits_{i=1}^n
\lambda_i\widehat{\Phi}(ye_i)=\widehat{\Phi}(ya).$ Let $a\in
Z_+(M)$ and $\{a_n\}$ be a sequence of simple elements from
$Z_+(M)$  with $a_n\uparrow a.$ Then $\Psi(a_n) \uparrow \Psi(a),$
$ya_n \uparrow ya,$ and $\widehat{\Phi} (ya_n)\uparrow
\widehat{\Phi}(ya)$ (see the proof of Proposition \ref{3.6.}).
Hence, $\Psi(a)=\widehat{\Phi}(ya)$ for all $a\in Z_+(M).$ Now
using the linearity of traces $\Psi$ and $\Phi,$ we obtain
$\Psi(a)=\widehat{\Phi}(ya)$ for all $a\in M.$

Furthermore, due to  Propositions \ref{3.2.}$(iv)$ and \ref{3.6.}
we get
$$
\Psi(x)=\Psi(\Phi_M(x))= \widehat{\Phi} (y\Phi_M(x))=
\widehat{\Phi}(\widehat{\Phi}_M(yx))=\widehat{\Phi}(yx)
$$
for all $x\in M.$
\end{proof}

\begin{rem} If $\Psi$ is a normal $F_\mathbb{C}$-valued trace on
$M$ and $\Psi\ll \Phi,$ then $\Psi$ possesses the Maharam
property.
\end{rem}

In fact, by Theorem \ref{3.7.}, $\Psi(x)=\widehat{\Phi}(yx)$ for
all $x\in M$ where $y\in L_+^1(M,\Phi)\cap S(Z(M)).$ Let ~$0\neq
x\in M_+,~ f\leq \Psi(x),~f\in S_+(\mathcal{B}),~g\in
S_+(\mathcal{B}),~g\Psi(x)=s(\Psi(x).$ Set $h=gf,~
 z=\psi^{-1}(h),~a=zx.$ Then $0\leq h \leq g\Psi(x)=s(\Psi(x)) \leq
\mathbf{1}_F,~0\leq z\leq \mathbf{1},~0\leq a\leq x$ and
$$\Psi(a)=\widehat{\Phi}(ya)=\widehat{\Phi}(zyx)=\psi(z)\widehat{\Phi}(yx)=h\Psi(x)=fs(\Psi(x))=f.$$

\end{document}